\long\global\def\C#1\F{{}}

\C
Submitted to arxiv on October 6??, 2017
\F

\documentclass[authoryear,preprint,12pt]{elsarticle} 

\usepackage[usenames,dvipsnames]{color}
\usepackage{amsmath, amsthm, amssymb, enumerate}
\biboptions{longnamesfirst,comma} 

\usepackage[normalem]{ulem}

\usepackage[displaymath, mathlines]{lineno} 
\usepackage{tikz}
\usetikzlibrary{decorations,decorations.pathmorphing,decorations.pathreplacing}
\usetikzlibrary{arrows,backgrounds,calc}
\usetikzlibrary{shapes.symbols,shapes.geometric}

\usetikzlibrary{decorations.markings}
 \tikzset{->-/.style={decoration={
  markings,
  mark=at position #1 with {\arrow{triangle 45}}}, 
  postaction={decorate}}}

\usepackage{enumerate}

\def\st{{\;\vrule height9pt width0.9pt depth2.5pt\;}}

\newcommand{\thr}{\tau} 
\newcommand{\src}{s}    
\newcommand{\snk}{t}    

\newcommand{\specialPrm}[2]{P_{\mathrm{#1}}^{#2}}
\newcommand{\Plo}[1]{\specialPrm{LO}{#1}} 

\newcommand{\argmax}{\operatornamewithlimits{argmax}}
\newcommand{\Pwo}[1]{\specialPrm{WO}{#1}} 
\newcommand{\Pso}[1]{\specialPrm{SO}{#1}} 
\newcommand{\Pio}[1]{\specialPrm{IO}{#1}} 
\newcommand{\Pcor}[1]{\specialPrm{COR}{#1}} 

\newcommand{\specialDrm}[2]{D_{\mathrm{#1}}^{#2}}
\newcommand{\Dlo}[1]{\specialDrm{LO}{#1}} 
\newcommand{\Dwo}[1]{\specialDrm{WO}{#1}} 
\newcommand{\Dso}[1]{\specialDrm{SO}{#1}} 
\newcommand{\Dio}[1]{\specialDrm{IO}{#1}} 

\newcommand{\conv}{\mathop{\mathrm{conv}}}

\newtheorem{proposition}{Proposition}

\newtheorem{theorem}{Theorem}
\newtheorem{lemma}{Lemma}

\newdefinition{defn}{Definition}
\newdefinition{exam}{Example}
\newdefinition{remark}{Remark}

\newproof{proo}{Proof}
\newproof{prof}{Proof of Theorem}

\begin{document}

\begin{frontmatter}

\title{Extended Formulations for Order Polytopes \\ through Network Flows}

\author{Clintin P. Davis-Stober}
\ead{cstober2@gmail.com}
\address{University of Missouri, Columbia.}

\author{Jean-Paul Doignon\footnote{Corresponding author: Jean-Paul Doignon}}
\ead{doignon@ulb.ac.be}
\address{Universit\'e Libre de Bruxelles, Belgium}

\author{Samuel Fiorini}
\ead{sfiorini@ulb.ac.be}
\address{Universit\'e Libre de Bruxelles, Belgium.}

\author{Fran\c{c}ois Glineur}
\ead{glineur@core.ucl.ac.be}
\address{Universit\'e Catholique de Louvain, Belgium.}

\author{Michel Regenwetter}
\ead{regenwet@uiuc.edu}
\address{University of Illinois at Urbana-Champaign. \vspace{-10mm}}

\begin{linenomath*}
\begin{abstract} 
Mathematical psychology has a long tradition of modeling probabilistic choice via distribution-free random utility models and associated random preference models.  
For such models, the predicted choice probabilities often form a bounded and convex polyhedral set, or polytope.  Polyhedral combinatorics have thus played a key role in studying the mathematical structure of these models.  
However, standard methods for characterizing the polytopes of such models are subject to a combinatorial explosion in complexity as the number of choice alternatives increases.  
Specifically, this is the case for random preference models based on linear, weak, semi- and interval orders.
For these, a complete, linear description of the polytope   
is currently known only for, at most, 5--8 choice alternatives. 
We leverage the method of extended formulations to break through those boundaries. For each of the four types of preferences, we build an  
appropriate network, and show that the associated network flow polytope provides an extended formulation of the polytope of the 
choice model.  This extended formulation has a simple linear description that is more parsimonious than descriptions obtained by standard methods for large numbers of choice alternatives. 
The result is a computationally less demanding way of testing the probabilistic choice model on data.  
We  sketch how the latter interfaces with recent developments in contemporary statistics.
\end{abstract}
\end{linenomath*}

\end{frontmatter}



\section{Introduction}\label{sec: Intro} 

Much of the literature on choice behavior in the social sciences centers on modeling the unobservable hypothetical preferences or cognitive processes underlying observable choice behavior. This endeavor encounters a huge hurdle in empirical applications: How can models accommodate heterogeneity in choice behavior across individuals as well as within any given individual?  To this end, researchers have employed a variety of probabilistic modeling approaches.  Prominent examples include computational, stochastic process models that mimic hypothetical cognitive processes, such as the well-known diffusion model \citep[e.g.,][]{ratcliff2004comparison, ratcliff1998modeling} and the linear ballistic accumulator model \citep{brown2008simplest,trueblood2014multiattribute}.  Other approaches include  multinomial processing tree models \citep[see, e.g.,][for a review]{erdfelder2009multinomial}. 
Rather than model latent cognitive processes, we concentrate on the more abstract notion of latent preferences.  
In so doing, we consider models that use a minimum of 
mathematical assumptions and that delineate large classes of theories of choice behavior. 
This approach is grounded in a long tradition of axiomatization and axiom testing in mathematical psychology \citep[see, e.g.,][]{luce2000utility} and naturally interfaces with  contemporary statistical methods \citep{davis2011shift,davis2015individual}.  Evaluating choice data against such models allows for strong inferences regarding the latent preferences that give rise to choice behavior.
Here, we tackle challenges with the mathematical characterization of such models.

Specifically, we consider classes of choice theories based upon four types of transitive binary preference relations: linear orders (rankings without ties), weak orders (rankings with or without ties), semiorders (partial rankings up to a constant threshold of discrimination), and interval orders (partial rankings up to a stimulus-dependent threshold of discrimination). For each of these algebraic representations of preferences, we consider probabilistic choice induced by an (unspecified) probability distribution over permissible preference states, i.e., over relations of the selected type. These models are also known in the literature as  \textsl{random preference models}, \textsl{mixture models}, or \textsl{random relation models}. 
Because each of these models has been shown to have a random utility representation, our results also directly
apply to characterizing various kinds of distribution-free random utility models, as well as random function models \citep[see, e.g.,][]{Regenwetter_Marley_2001}. 

Among the most important applications of these models are studies aimed at testing whether decision makers have transitive preferences (e.g., because they employ compensatory decision strategies by which they coherently trade-off between competing decision attributes) or whether they violate transitivity (say, because they employ simple heuristics that cause `incoherent' preferences). It is well-known \citep[][]{Roberts_1979} that, for finitely many choice options, complete asymmetric binary preferences are transitive if and only if they are strict linear orders. Dropping the requirement of complete preference,   asymmetric 
binary preferences on a finite set of choice options are negatively transitive (hence also transitive) if and only if they are strict weak orders \citep[][]{Roberts_1979}.  Testing random preference models on laboratory data faces two main hurdles: i)~Because the model ranges
of these models form convex polytopes 
(that is, bounded convex polyhedral sets as in \citealt{Ziegler98}),
they require order-constrained statistical methods; ii)~The order constraints 
(geometrically, the linear inequalities defining facets of the polytope) are  fully known only when the number of choice alternatives is small.  Recent developments  have essentially solved the first challenge \citep[see][]{Davis_Stober_2009,karabatsos2006bayesian, klugkist2007bayes,Myung_Karabatsos_Iverson_2005, Silvapulle_Sen_2005}. 

Using some of these approaches, \cite{Regenwetter_Dana_Stober_2011a} provided the first empirical and  statistically adequate test of the ``linear ordering model''  in the literature. 
Following \cite{Tversky_1969}, they used five choice alternatives per stimulus set.  \cite{Regenwetter_Stober_2012}, who reported the first empirical test of the ``weak order model,'' 
also used  five choice alternatives per stimulus set, in this case because the mathematical structure of the weak order polytope is not yet fully understood for larger numbers of choice alternatives. 
\cite{Regenwetter_Stober_2011} used four choice alternatives to test the ``semiorder model'' and the ``interval order model'' because they did not know a full description of the corresponding polytopes for more than four options. 

In this article, we leverage extended formulations. 
We study network flow polytopes to  characterize  random preference models of linear, weak, semi-, and interval orders in novel ways.  
This approach generates a simpler, more parsimonious description of the corresponding random preference model and
 allows empirical researchers to study these models for  larger numbers of choice alternatives than previously possible.  

The paper proceeds as follows. First, we review key facts about linear, weak, semi-, and interval orders. Then, we discuss the convex polytopes associated with the 
random preference models based on each of these four types of preference relations and the obstacles faced by traditional mathematical methods for characterizing their model ranges.
The following two sections overcome these obstacles and investigate extended formulations, especially those based on network flows, and associated 
network flow polytopes. We then discuss how these network flow polytopes interface with some contemporary developments in statistics.
We end with a summary and detail future directions of this work.

\section{Models of Pairwise Preferences and their Numerical Representations}
\label{sec: Ordinal} 

Throughout the paper, $S$ denotes a finite set of $n$ choice alternatives (thus, $|S|=n$).  The preferences of a decision maker among the alternatives are cast as a relation $R$ on $S$, with $i \,R\, j$ meaning that
the decision maker likes $i$ strictly less than $j$ (hence strictly prefers $j$ over $i$).
In a simple case, the relation $R$ is a \textsl{linear order}, that is an irreflexive, transitive and complete relation (for a definition of these and similar terms, see for instance 
 \citealt[Chapter~1]{Fishburn1985}, or 
\citealt*[Chapter~3]{PirlotVincke1997}).
As is well-known, such an ideal situation occurs when preferences reflect comparisons of utility values 
and no two choice alternatives have the same utility.

\begin{proposition}[Linear order]\label{pro: linear order}
The relation $R$ on the set $S$ of alternatives is a linear order if and only if there exists an injective mapping $u:\;S \to \mathbb{R}$ such that, for all $i$, $j$ in $S$:
\begin{equation}\label{eqn: linear order}
i\,R\,j \;\iff\; u(i) < u(j).
\end{equation}
\end{proposition}

If we drop the injectivity condition in Proposition~\ref{pro: linear order}, thus making room for alternatives having equal utility, we get a \textsl{(strict) weak order} (an asymmetric and negatively transitive relation). 

\begin{proposition}[Weak order]\label{pro: weak order}
The relation $R$ on the set $S$ of alternatives is a weak order if and only if there exists a mapping $u:\;S \to \mathbb{R}$ such that, for all $i$, $j$ in $S$:
\begin{equation}\label{eqn: weak order}
i\,R\,j \;\iff\; u(i) < u(j).
\end{equation}
\end{proposition}

In many situations, comparing utility values defined as real numbers is not realistic.  According to an extended viewpoint, any alternative $i$ from $S$ is assigned a range of values taking the form of a real interval $[\ell(i),h(i)]$ in such a way that $i\,R\,j$ holds exactly if the interval $[\ell(i),h(i)]$ lies entirely before the interval $[\ell(j),h(j)]$ (see Figure~\ref{fig: interval order} for a geometric illustration). The relation $R$ is then called an \textsl{interval order} on $S$.  
Here, one may interpret $\ell(i)$ as the lower utility value of $i$, $h(i)$ as the upper utility value, and $h(i)-\ell(i)$ as a perceptual \textsl{threshold} for utilities.
We call such a pair of mappings $(\ell,h)$ a \textsl{representation} of $R$.

\begin{figure}[htbp]
\begin{center}
\begin{tikzpicture} [>=triangle 45]

\draw[arrows=|-|, thick] (3,0.2) node at +(0,0.4) {$\ell(i)$} -- (4.05,0.2) node at +(0,0.4) {$h(i)$};

\draw[arrows=|-|, thick] (6,0.2) node at +(0,0.4) {$\ell(j)$} -- (8.05,0.2) node at +(0,0.4) {$h(j)$};

\draw (0,-0.1) node at +(0,-0.2) {0} -- (0,0.1);

\draw[->] (0,0) -- (11.5,0) node at +(0.5,0) {$\mathbb R$};
\end{tikzpicture} 
\caption{\rm Geometric representation of $i \,R\, j$ in the case of an interval order $R$, thus: $i$ is less preferable than $j$ if and only if $h(i) < \ell(j)$.}
\label{fig: interval order}
\end{center}
\end{figure}
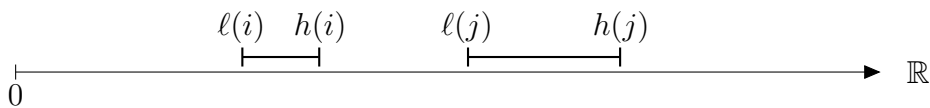
 
A combinatorial characterization of interval orders, due independently to \cite{Fishburn1970} and \cite{Mirkin1972} (but also included in the result on `bi-quasi-series' in \citealt{Ducamp_Falmagne69}), is as follows.

\begin{proposition}[Interval order]\label{pro: interval order}
The relation $R$ on the set $S$ of alternatives satisfies the following two conditions, for all $i$, $j$, $i'$, $j'$ in $S$:
\begin{gather}
\text{not~}(i \,R\, i);\\
(i \,R\, j \text{ and } i' \,R\, j') \;\implies\; (i \,R\, j' \text{ or } i' \,R\, j)
\end{gather}
if and only if there exist two mappings $\ell,\,h:\;S \to \mathbb{R}$ such that the following two conditions hold, for all $i$, $j$ in $S$:
\begin{gather}\label{eqn: interval order}
\ell(i) \le h(i);\\
\label{eqn: interval order bis}
i\,R\,j \;\iff\; h(i) < \ell(j).
\end{gather}
\end{proposition}

When the threshold $h(i)-\ell(i)$ in Proposition~\ref{pro: interval order} is required to be independent of the alternative $i$, we get the semiorder model, and say that the relation $R$ is a \textsl{semiorder} on $S$.  The latter model appears in \cite{Wiener_1915}, \cite{Armstrong1939} and \cite{Luce1956} and was given its name by the last author.  
The following characterization is due to \cite{{ScottSuppes1958}}.

\begin{proposition}[Semiorder]\label{pro: semiorder}
The relation $R$ on the set $S$ of alternatives satisfies the following three conditions, for all $i$, $j$, $i'$, $j'$, $i''$ in $S$:
\begin{gather*}
\text{not~}(i \,R\, i);\\
(i \,R\, j \text{ and } i' \,R\, j') \;\implies\; (i \,R\, j' \text{ or } i' \,R\, j);\\
(i \,R\, i' \text{ and } i' \,R\, i'') \;\implies\; (i \,R\, j \text{ or } j \,R\, i'')
\end{gather*}
if and only if there exist two mappings $\ell,h:\;S \to \mathbb{R}$ and a nonnegative real number $r$ such that the following two conditions hold, for all $i$, $j$ in $S$:
\begin{gather}\label{eqn: semiorder}
i\,R\,j \;\iff\; h(i) < \ell(j);\\
h(i) - \ell(i) = r.
\end{gather}
\end{proposition}

\section{Random Preference Models and their Associated Polytopes}

Relations like those of Propositions~\ref{pro: linear order}--\ref{pro: semiorder} can be used to model the decision maker's preferences at a particular moment.  
In general, the decision maker may be uncertain about her preferences and may probabilistically sample a preference from some collection of preference states, when required to make decisions.
We assume all permissible preference states obey the model selected.  
Thus, at any time we may interview the decision maker about any pair $(i,j)$ of alternatives, and collect an answer reflecting her currently sampled preference.  
As a consequence, repeated interviews will lead to frequencies of having $i \,R\, j$ (or not).  On the theoretical side, 
this leads us to posit a probability distribution $P$ on the set $\mathcal{LO}_S$ of all linear orders on $S$---to take this model as an example; similar considerations of course hold in the other three models, for the collection $\mathcal{WO}_S$ of all weak orders,  $\mathcal{IO}_S$ of all interval orders, or $\mathcal{SO}_S$ of all semiorders.  With $P(R)$ denoting the probability that the decision maker's (currently sampled) relation is $R$, we assume that the probability that she prefers alternative $j$ over alternative $i$ is given by the following formula:
\begin{equation}\label{eqn: p(i,j)}
p_{(i,j)} \;=\; \sum\;\{P(R) \st R\in\mathcal{LO}_S \text{ with } i \,R\, j \}.
\end{equation}
Not all collections of real values~$\{p_{(i,j)} \st (i,j)\in S \times S\}$ can be produced in this way, for a probability distribution $P$ on $\mathcal{LO}_S$.  For instance, Equation~\eqref{eqn: p(i,j)} implies $p_{(i,j)} \ge 0$ and, for any $i\neq j$,  $p_{(i,j)} + p_{(j,i)} = 1$.  Other necessary conditions are known, but necessary and sufficient conditions are known only for small values of $n$.
There are doubts that a tractable set of necessary and sufficient conditions  could be found that is valid for all $n$ (because the `linear ordering problem' is NP-hard, see \citealp{Garey_Johnson79}).  To get a better understanding of the collections $\{p_{(i,j)} \st (i,j)\in S \times S\}$ 
that satisfy Equation~\eqref{eqn: p(i,j)} for at least one
probability distribution $P$ on $\mathcal{LO}_S$, we now develop a geometric interpretation and we do the same later for the other three models.  

Each linear order on the set $S$, being a collection of pairs of alternatives of $S$, is a subset of $S \times S$. 
Because linear orders (as well as the other orders we consider in this paper) are irreflexive, they never contain a pair of the form $(i,i)$, with $i \in S$.  
This prompts us to discard all pairs $(i,i)$, for $i \in S$.  Let 
$S \star S := (S \times S) \setminus \{(i,i) \st i \in S \}$.
Consider the space $\mathbb{R}^{S \star S}$, in which a point $x$ has a coordinate $x_{(i,j)}$ for each pair $(i,j)$ of 
distinct
alternatives from $S$.  To any linear order $R$ on $S$, associate the point $x^R$ with coordinates given by
\begin{equation}
\left(x^R\right)_{(i,j)} \;:=\; 
\begin{cases}
1 &\text{if }i \,R\, j,\\
0 &\text{otherwise}.
\end{cases}
\end{equation}
In other words, $x^R$ is the \textsl{characteristic vector} of the linear order $R$.  The convex hull in $\mathbb{R}^{S \star S}$ of the characteristic vectors of all linear orders on $S$, that is the set of all convex combinations of those vectors (linear combinations with nonnegative coefficients summing to one), is the \textsl{linear order polytope of} $S$, which we denote
$\Plo S$:
\begin{equation}
\Plo S \;:=\; \conv\;\{x^R \st R \in \mathcal{LO}_S \}.
\end{equation}
A collection $\{p_{(i,j)} \st (i,j)\in S \star S\}$ satisfies Equation~\eqref{eqn: p(i,j)} for 
at least one probability distribution $P$ on $\mathcal{LO}_S$ if and only if the point $p$ lies in the linear order polytope $\Plo S$.  Thus, to characterize such collections, we would like a criterion for a point $x$ from $\mathbb{R}^{S \star S}$ to lie in $\Plo S$.  
A classical result (see, for instance, \citealt{Ziegler98}) states that any polytope admits a linear description, that is a finite system of linear equations and inequalities whose solution set is exactly the polytope.  Even if we insist that the system has a minimum number of equations and inequalities, 
it is generally not unique; however, the number of linear equations is then the codimension of the polytope, and there is exactly one inequality per facet of the polytope.  
In the case of the linear order polytope $\Plo S$, the number of equations on $\mathbb{R}^{S \star S}$ equals 
$\binom{n}{2} = n(n-1)/2$ (those being the already mentioned $p_{(i,j)} + p_{(j,i)} = 1$), 
but no precise estimation seems to be available on the number of inequalities.  As far as we know, a complete linear description of $\Plo S$ is lacking when $|S| = n > 7$. For $n = 8$, the number of vertices is $8! = 40{,}320$ but only a lower bound on the number of facets is available. This lower bound is huge: the polytope has at least $488{,}602{,}996$ facets \citep{Christof_Reinelt96}.   For other results on the linear order polytope, we refer the reader to  
\cite{Doignon_Fiorini_Joret2006, Fiorini_2006a, Fiorini_2006b, Boyd_Pulleyblank_2009, Doignon_Fiorini_Joret2009, Oswald_Reinelt_Seitz_2009, Marti_Reinelt_2011} and their references.
For an application of the linear order model to test transitivity of preferences on empirical data, see \cite{Regenwetter_Dana_Stober_2010, Regenwetter_Dana_Stober_Guo_2011, Regenwetter_Dana_Stober_2011a}.

The \textsl{weak order polytope of} $S$, which we denote by
$\Pwo S$, is similarly defined in $\mathbb{R}^{S \star S}$:
\begin{equation}
\Pwo S \;:=\; \conv\;\{x^R \st R \in \mathcal{WO}_S \}.
\end{equation}

Note in passing that several publications \cite[e.g.,][]{Doignon_Fiorini2002,Fiorini_Fishburn2004} consider `complete weak orders.'  The latter are transitive and complete relations, in other words the inverses of the complements of weak orders.  The resulting `complete weak order polytope' is the image of our polytope $\Pwo S$ by the affine transformation $\mathbb{R}^{S \star S} \to \mathbb{R}^{S \star S}:\; x \mapsto x'$ with $x'_{(i,j)} = 1 - x_{(j,i)}$.  Hence, all the results about one of these polytopes have a counterpart for the other.

The Ph.D.\ dissertation of \cite{Fiorini_thesis} contains an overview of results about the two order polytopes we just introduced, including a complete description of $P_{WO}^S$ for $|S|=n=4$.
To our knowledge, a complete description of $P_{WO}^S$ is known only up to $n = 5$. For $n = 5$, the weak order polytope $P_{WO}^S$ has $75{,}834$ facets and $541$ vertices \citep{Regenwetter_Stober_2012}.

Two other polytopes play a role in our study.  They are the \textsl{interval order polytope of} $S$
\begin{equation}
\Pio S \;:=\; \conv\;\{x^R \st R \in \mathcal{IO}_S \}
\end{equation}
and the semiorder polytope of $S$
\begin{equation}
\Pso S \;:=\; \conv\;\{x^R \st R \in \mathcal{SO}_S \}.
\end{equation}
The interval order polytope is investigated by \cite{Muller_Schulz95},   
the semiorder polytope by \cite{Suck95}
and both of them by \cite{Doignon_Rexhep_2016}.   
For $n= 4$, \cite{Regenwetter_Stober_2011} reports that the interval order polytope has $207$ vertices and $191$ facets, while the figures for the semiorder polytope are $183$ and $563$.

\section{Network Flow Representations of Order Polytopes}\label{sec: Network}

We now move to a third way of investigating our four order polytopes. 
Many applications in psychology and economics aim to evaluate random preference models on human subjects choice data from laboratory studies. 
This typically involves optimizing concave functions (e.g., a log-likelihood function) subject to the constraint that the choice probabilities belong to the polytope in question. Such optimizations are straightforward when the geometric structure of the polytope is fully known. The extended formulations that we now consider
make it possible to solve some of these optimization goals in cases where obtaining a complete geometric representation of the polytope in question is computationally  expensive.

\subsection{The concept of an extended formulation}
In the previous section, we described how testing whether a collection of choice probabilities satisfies the linear order model amounts to checking whether a given point lies in the linear order polytope $\Plo S$ (we again base our exposition on this model), which can be done by checking the validity of each linear equation and inequality in its description).  Unfortunately, 
unless we limit ourselves to small enough $n$ the latter task is 
made difficult or even impossible by the inherent intricacy, or sheer unavailability, of a linear description of $\Plo S$.  In this section, we sketch the main steps to overcome this difficulty.  We provide technical constructions and proofs in the next section.  

The main idea is to work with another, ad hoc polytope that projects in some specific way on $\Plo S$ 
(more precisely: there is an affine transformation mapping the ad hoc polytope onto $\Plo S$).
Any linear description of such an ad hoc polytope is called an \textsl{extended formulation} of $\Plo S$.
For a survey of the notion of an extended formulation, see 
\cite{Conforti_Cornuejols_Zambelli2010}, \cite{Kaibel11} and \cite{Wolsey11}.
For the four order polytopes we study in this paper, the 
ad hoc polytope happens to be a ``flow polytope'', so we first define the latter notion and review some known facts.

\subsection{The concept of a network flow polytope}
Let $D = (N,A)$ be a network with node set $N$ and arc set $A$ 
\cite[for network terminology, see, for instance,][]{KorteVygen08}.
In $N$, we designate two special nodes: the \textsl{source} node $\src$ and the \textsl{sink} node $\snk$. 

For a node $v$, we denote the sets of arcs leaving and entering $v$ by $\delta^+(v)$ and $\delta^-(v)$, respectively. Formally,
\begin{eqnarray*}
\delta^+(v) &:= &\{a \in A \st \exists w \in N : a = (v,w)\},\\
\delta^-(v) &:= &\{a \in A \st \exists t \in N : a = (t,v)\}.
\end{eqnarray*}

Consider a set $B$ of arcs in $A$. We encode $B$ by its characteristic vector $\chi^B$ in $\mathbb{R}^{A}$, defined by letting $\chi^B_a := 1$ if $a \in B$ and $\chi^B_a := 0$ if $a \in A \setminus B$ (we notate the characteristic vectors $x^R$ and $\chi^B$ differently to emphasize that they refer to the distinct master sets $S \star S$ and $A$, respectively). For a vector $\Phi$ in $\mathbb{R}^{A}$, we define the number
\begin{equation}
\Phi(B) := \sum_{a \in B} \Phi_a.
\end{equation}

A \textsl{flow} in $D$ is a vector $\Phi$ from $\mathbb{R}^{A}$, associating a nonnegative number $\Phi_a$ to each arc $a$ of the network, such that the out-flow $\Phi(\delta^{+}(v))$ equals the in-flow $\Phi(\delta^{-}(v))$ at each node $v$ distinct from the source node $\src$ and the sink node $\snk$. The \textsl{value} of a flow $\Phi$ equals $\Phi(\delta^{+}(\src)) - \Phi(\delta^{-}(\src))$, that is, the net out-flow at the source node. A flow $\Phi$ is said to be \textsl{integral} if $\Phi_a$ is an integer, for all arcs $a$.

In the context of this work, we assume that the network $D$ is acyclic. We define the \textsl{flow polytope} $F = F(D)$ of the network $D$ as the polytope whose vertices are the characteristic vectors of the sets of arcs of all the $\src$--$\snk$ (directed) paths in $D$.  Any such characteristic vector is an integral $\src$--$\snk$ flow in $D$ of value $1$.  Because we assume that $D$ is acyclic, the converse also holds: any vector $\Phi$ from $\mathbb{R}^{A}$ which is an integral $\src$--$\snk$ flow in $D$ of value $1$ is the characteristic vector of an $\src$--$\snk$ path in $D$.  A complete linear description of $F(D)$ is as follows \cite[see for instance][Theorem~8.8]{KorteVygen08}---we call it the \textsl{canonical description} of the flow polytope $F(D)$:

\begin{equation}\label{eqn: flow polytope}
F(D) =
\left\{\Phi \in \mathbb{R}^{A} \;
\vrule height23pt width1pt depth14pt \; 
\begin{array}{rcl@{\quad}l}
\Phi(\delta^{+}(v)) - \Phi(\delta^{-}(v)) &= &0, &\forall v \in N \setminus \{\src,\snk\},\\
\Phi(\delta^{+}(\src)) - \Phi(\delta^{-}(\src)) &= &1,\\
\Phi_a &\geqslant &0, &\forall a \in A
\end{array}
\right\}.
\end{equation}

Following the current practice in extended formulations, we define the \textsl{size} of a linear description as the total number of inequalities in the description, thus disregarding the number of variables and equalities, as well as the bit complexity of the coefficients. Hence, the size of the canonical description of $F(D)$ is $|A|$.

Again, let $S$ denote the set of alternatives, with $|S| = n$.  We now show that three of our order polytopes, namely the linear order, weak order and interval order polytopes, share the following property for some constant $c$ associated with each family of polytopes: For each of these polytopes, there exists a flow polytope $F$ with a linear description of size $O(c^n)$ and a projection from $\mathbb{R}^{A}$ to $\mathbb{R}^{S \star S}$ that maps the flow polytope $F$ to the considered order polytope.  We explicitly build this flow polytope, which is thus an extended formulation.  We also build a flow polytope providing an extended formulation for the semiorder polytope with, surprisingly, canonical description size in $\Omega(n!)$, which is not $O(c^n)$ for any constant $c$.

As we discuss later in Section~\ref{sec:optimality}, it follows from recent results that all our extended formulations are size-optimal, except perhaps that for the semiorder polytope. More precisely, any extended formulation for any of our order polytopes for a set of alternatives of size $n$ has size at least $(3/2)^{\lfloor n/2 \rfloor}$.

In order to specify the networks, the key idea is to use ordinal representations of the relations involved. We start with weak orders.

\subsection{The case of weak orders}  
Remember from Proposition~\ref{pro: weak order} that a relation $R$ on $S$ is a weak order exactly if there exists a mapping $u$ from $S$ to $\mathbb{R}$ such that
\begin{equation}
i\,R\,j \iff u(i) < u(j).
\end{equation} 
For each $\thr$ in $\mathbb{R}$, we let
\begin{equation}
X(\thr) := \{i \in S \st u(i) < \thr\}.
\end{equation} 

Now, imagine increasing the \textsl{level} $\thr$ continuously from $\thr = \min u(S)$ to $\thr = \max u(S) + 1$. Thus, we observe a finite sequence of distinct sets $X_0$, $X_1$, \ldots, $X_m$, starting with the empty set $\varnothing$ and ending with the entire set $S$. We call this sequence the \textsl{profile of the weak order\/} $R$ (notice that the profile of the weak order $R$ does not depend on the representation of $R$, and also that it determines $R$). 
More formally, let $u(S) = \{\thr_0, \thr_1, \dots, \thr_{m-1}\}$ with $\thr_0 < \thr_1 < \dots < \thr_{m-1}$. For every $k$ in $\{0,1,\ldots,m-1\}$, set $X_k := \{i \in S \st u(i) < \thr_k\}$, and moreover $X_{m}:=S$.
The profile of the weak order $R$ forms a chain
\begin{equation}
\varnothing =: X_0 \subset X_1 \subset \cdots \subset X_m := S.
\end{equation} 
(Here and throughout, $\subset$ indicates strict inclusion.)

Now form the network whose nodes are the subsets of $S$, letting $\src := \varnothing$ and $\snk := S$, and whose arcs are the pairs $(X,Z)$ of subsets of $S$ such that $X \subset Z$.
We see that the profile of each weak order $R$ determines an $\src$--$\snk$ path in this network.
Conversely, any $\src$--$\snk$ path in this network determines a unique weak order. 
Because the two constructions correspond to mappings that are each other's inverses, 
we get a bijection from the set of $\src$--$\snk$ paths to the set of weak orders on $S$.
Later (in Subsection~\ref{sub_weak_order}), we prove that this bijection yields a projection from the flow polytope of the network we just built to the weak order polytope of $S$.

\subsection{The case of linear orders} 
We now turn to the linear orders on $S$.
In the above network (for weak orders), they correspond to the $\src$--$\snk$ paths  
that use only arcs $(X,Z)$ where 
$X \subset Z$ with the restriction that $|X|+1=|Z|$.  Because the $\src$--$\snk$ paths describing linear orders only use those arcs, we now delete all the other arcs from the network while keeping all the nodes.  In the resulting network, the $\src$--$\snk$ paths bijectively correspond to the linear orders on $S$.
We later infer a projection from the flow polytope of this network  
to the linear order polytope of $S$ (see Theorem~\ref{thm: proj_LO}).

\subsection{The case of interval orders} 
To create a network for interval orders, we use Proposition~\ref{pro: interval order} according to which
a relation $R$ on $S$ is an interval order when there exist two mappings $\ell$ and $h$ from $S$ to $\mathbb{R}$ such that $\ell(i) \leqslant h(i)$ for all $i \in  S$ and
\begin{equation}
i\,R\,j \iff h(i) < \ell(j). 
\end{equation}
In the following, we always assume that $\ell$ and $h$ are both one-to-one, and that the images of $S$ under $\ell$ and $h$ are disjoint. 
Why can we modify the mappings $\ell$ and $h$ in order to satisfy these two conditions?  Each value of $\ell$ is only constrained by Equations~\eqref{eqn: interval order} and \eqref{eqn: interval order bis} to lie in a convex subset of $\mathbb{R}$ which is determined by the values of $h$.  Moreover this convex set has more than one point.  Hence, we can always adjust values of $\ell$ in order to make $\ell$ one-to-one.  Similarly, we can next adjust values of $h$ in order to make $h$ one-to-one and at the same time $\ell(S)$ and $h(S)$ disjoint.

Each \textsl{level} $\thr$ from $\mathbb{R}$ determines two subsets of $S$ (for a geometric example, see Figure~\ref{fig semiorder intervals}, but, for the moment, ignore the last sentence in the caption):
\begin{eqnarray}\label{eqn: X(tau)}
X(\thr) & := &\{i \in S \st \ell(i) < \thr\},\\ 
\label{eqn: X(tau) bis}
Y(\thr) & := &\{i \in S \st h(i) < \thr\}.
\end{eqnarray}
Because $\ell(i) \le h(i)$ holds for all $i \in S$, we have $Y(\thr) \subseteq X(\thr)$ for every level $\thr$. When we continuously increase $\thr$ from $\thr = \min \ell(S)$ to $\thr = \max h(S) + 1$, we observe a finite sequence of distinct pairs of sets $(X_0,Y_0)$, $(X_1,Y_1)$, \ldots, $(X_m,Y_m)$ that forms the \textsl{profile of the interval order\/} $R$ (notice that the profile of an interval order $R$ technically depends on the chosen representation $\ell$, $h$ of $R$; however, each such profile determines $R$).
Our assumptions on the mappings $\ell$ and $h$ imply that any two consecutive pairs of sets $(X_k,Y_k)$, $(X_{k+1},Y_{k+1})$ satisfy the following three conditions: 
\begin{enumerate}
\item[$\cdot$] $X_k \subseteq X_{k+1}$,
\item[$\cdot$] $Y_k \subseteq Y_{k+1}$,
\item[$\cdot$] either
$\left\{
\begin{array}{l}
|X_{k+1}| = |X_k| + 1 \\[1mm]
|Y_{k+1}| = |Y_k|
\end{array}
\right.$
\quad or
$
\left\{
\begin{array}{l}
|X_{k+1}| = |X_k|\\[1mm]
|Y_{k+1}| = |Y_k| + 1.
\end{array}
\right.$
\end{enumerate}
Using these conditions as an inspiration, we define the network  
$\Dio S$ whose nodes are pairs $(X,Y)$ of subsets of $S$ with $Y \subseteq X$, with source node $\src := (\varnothing,\varnothing)$ and sink node $\snk := (S,S)$ (the details are given below, including the definition of the arcs). Each $\src$--$\snk$ path in this network encodes an interval order on $S$. Although several of these paths encode the same interval order, we obtain a projection from the flow polytope of $\Dlo S$ to the interval order polytope of $S$.

\subsection{The case of semiorders} 
The case of semiorders is a bit more involved.  Semiorders are particular interval orders, so they correspond to profiles of interval orders having an additional property.  We would like to build a new network whose $\src$--$\snk$ paths exactly correspond to the profiles of semiorders.  We obtain the nodes of the new network 
similarly as those of $\Dlo S$, but we store more information at each node.
The crucial property on which we rely is well known \cite[see for instance][]{Bogart_West_1999}:
an interval order is a semiorder if and only if it admits a representation by intervals such that no interval contains another interval unless they share an endpoint.  We now state this property 
formally.

\begin{proposition}\label{pro: no included interval}
A relation $R$ on the set $S$ is a semiorder if and only if  
there exist two mappings $\ell, h:\;S \to \mathbb{R}$ such that, for all $i$, $j$ in $S$:
\begin{gather}\label{eqn: no embedded intervals}
\ell(i) \le h(i);\\
i\,R\,j \;\iff\; h(i) < \ell(j);\\
\ell(i) < \ell(j) \implies h(i) \le h(j).
\end{gather}
\end{proposition}

Proposition~\ref{pro: no included interval} can be easily deduced from Proposition~\ref{pro: semiorder} (direct proofs also exist, see for instance \citealt{Bogart_West_1999}).

The network we build for semiorders has as nodes all triples $(X,Y,L)$ with $Y \subseteq X \subseteq S$ and with $L$ a linear order on $X \setminus Y$.  (Notice that when $X \setminus Y = \varnothing$, we must have $L=\varnothing$.)  The source and sink nodes are $\src := (\varnothing,\varnothing, \varnothing)$ and $\snk := (S,S, \varnothing)$, respectively.
To motivate the later specification of arcs, we show how each representation of a specific type of a semiorder generates a sequence of nodes that eventually becomes an $\src$--$\snk$ path.  
Given a representation $\ell$, $h$ as in Proposition~\ref{pro: no included interval} for a semiorder $R$, first adjust the values of $\ell$, $h$ to get a representation where $\ell$, $h$ are one-to-one and where $\ell(S) \cap h(S) = \varnothing$.
Then, for any real number $\thr$ with $\min \ell(S) \leqslant \thr \leqslant \max h(S)+1$, take the node obtained by setting $X := X (\thr)$, $Y := Y (\thr)$ as in Equations~\eqref{eqn: X(tau)} and \eqref{eqn: X(tau) bis}, and for $i,j\in X \setminus Y$ letting $i \,L\, j$ exactly when $\ell(i) < \ell(j)$.  An illustration is given in Figure~\ref{fig semiorder intervals}.  Varying the value of $\thr$, we thus generate a sequence of distinct nodes.  We later specify the  arcs so that each such sequence becomes an $\src$--$\snk$ path, and each $\src$--$\snk$ path is such a sequence.  Furthermore, the resulting flow polytope provides an extended formulation of the semiorder polytope.  The details are provided in Subsection~\ref{Semiorder polytopes}.  In particular, we show that the canonical description of this flow polytope grows faster in size than the extended formulations of our three other order polytopes. In the case of the semiorder polytope, we get an extended formulation of size $2^{\Theta(n \log n)}$, while for the other order polytopes, we obtain extended formulations of size $2^{\Theta(n)}$.

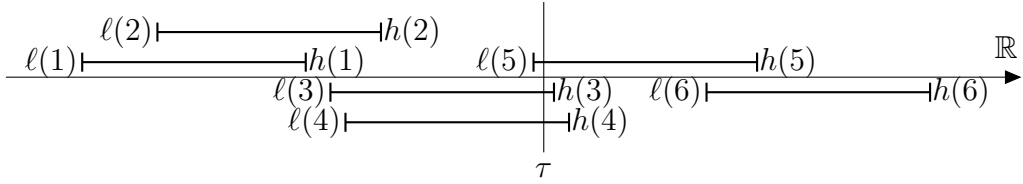
\begin{figure}[htbp]
\begin{center}
\begin{tikzpicture} [>=triangle 45]
\draw[arrows=|-|, thick] (0,0.2) node at +(-0.4,0) {$\ell(1)$} -- ++(3.0,0) node at +(0.4,0) {$h(1)$};
\draw[arrows=|-|, thick] (1,0.6) node at +(-0.4,0) {$\ell(2)$} -- ++(3.0,0) node at +(0.4,0) {$h(2)$};
\draw[arrows=|-|, thick] (3.3,-0.2) node at +(-0.4,0) {$\ell(3)$} -- ++(3.0,0) node at +(0.4,0) {$h(3)$};
\draw[arrows=|-|, thick] (3.5,-0.6) node at +(-0.4,0) {$\ell(4)$} -- ++(3.0,0) node at +(0.4,0) {$h(4)$};
\draw[arrows=|-|, thick] (6,0.2) node at +(-0.4,0) {$\ell(5)$} -- ++(3.0,0) node at +(0.4,0) {$h(5)$};
\draw[arrows=|-|, thick] (8.3,-0.2) node at +(-0.4,0) {$\ell(6)$} -- ++(3.0,0) node at +(0.4,0) {$h(6)$};

\draw (6.15,-1) node[below] {$\thr$} -- (6.15,1);

\draw[->] (-1,0) -- (12.5,0) node at +(-0.2,0.4) {$\mathbb R$};
\end{tikzpicture} 
\caption{\rm Illustration for the network node coming from the level $\thr$, in the case of interval orders: $X(\thr)=\{1,2,3,4,5\}$, $Y(\thr)=\{1,2\}$ (cf.~Equations~\eqref{eqn: X(tau)} and \eqref{eqn: X(tau) bis}). In the case of semiorders, we add the linear order $L=\{(3,4),\, (4,5),\, (3,5)\}$, because $\ell(3) < \ell(4) < \ell(5)$.}
\label{fig semiorder intervals}
\end{center}
\end{figure}

\section{Customizing Extended Formulations as Network Flow Polytopes}
\label{sec: Customizing Extended Formulations}

In the next four subsections, we give detailed constructions and proofs of the results announced in Section~\ref{sec: Network}.

\subsection{The case of the linear order polytope}

The network $\Dlo S = (N,A)$ we use for the case of linear orders of $S$ is defined by
\begin{eqnarray*}
N & := & \{X \st X \subseteq S\},\\ 
A & := & \{(X,Z)\in N \times N  \st X \subset Z,\ |Z| = |X| + 1\}
\end{eqnarray*}
(in technical terms, the network $\Dlo S$ is the ``Hasse diagram'' of the ``Boolean lattice'' of $S$). 
Thus, $|N| = 2^n$ and $|A| = n\,2^{n-1}$. The source and sink nodes are $\src := \varnothing$ and $\snk := S$, respectively. We now prove that the resulting flow polytope $F(\Dlo S)$ projects in a natural way onto the linear order polytope $\Plo S$.

\begin{theorem}
\label{thm: proj_LO}
Let $\pi$ be the projection from $\mathbb{R}^{A}$ to $\mathbb{R}^{S \star S}$, mapping a point $\Phi \in \mathbb{R}^A$ to the point $x \in \mathbb{R}^{S \star S}$ specified for all distinct
$i$, $j$ in $S$ by
\begin{equation}
\label{eqn: proj_LO}
x_{(i,j)} := \sum \{\Phi_a \st a = (X,Z) \in A,\ i \in X \textrm{ and } j \in Z \setminus X\}.
\end{equation}
The projection $\pi$ maps the flow polytope $F(\Dlo S)$ onto the linear order polytope $\Plo S$.  As a consequence, the canonical description of the flow polytope $F(\Dlo S)$, as in Equation~\eqref{eqn: flow polytope}, forms an extended formulation of $\Plo S$ 
with description size $n \, 2^{n-1}$.
\end{theorem}

\begin{proof}
It suffices to prove that the image of the vertex set of $F(\Dlo S)$ under the projection $\pi$ is exactly the vertex set of $P_{LO}^S$. 

First, consider a vertex $v$ of $P_{LO}^S$. Then $v$ is the characteristic vector of a linear order $R$ on $S$. By considering any ordinal representation $u :~ S \to \mathbb{R}$ of $R$ and the corresponding profile, we get an $\src$--$\snk$ path in $\Dlo S$, say $P$. The characteristic vector $\chi^{A(P)}$ of the set of arcs of this path is a vertex of $F(\Dlo S)$. 
We establish that $\pi(\chi^{A(P)}) = v$, i.e., $\Phi=\chi^{A(P)}$ in \eqref{eqn: proj_LO} 
yields $x_{(i,j)}=v_{(i,j)}$ for all 
distinct 
alternatives $i$ and $j$ of $S$.  
 
We evaluate the right-hand side of \eqref{eqn: proj_LO} at $\Phi := \chi^{A(P)}$. 
The conditions $i \in X$ and $j\in Z \setminus X$ imply $u(i) < u(j)$, and so $i \,R\, j$.  Moreover, they also imply $Z = X \cup \{j\}$, because here $X \subset Z$ and $|Z|=|X|+1$.  By definition of the characteristic vector $\Phi = \chi^{A(P)}$, the value $\Phi_a$ for any arc $a$ is $0$ or $1$.  From what we just proved, $\Phi_a=1$ can occur in \eqref{eqn: proj_LO} only for the unique arc $(X,Z)$ such that $Z = \{k\in S \st u(k) \le u(j)\}$ and $X = Z \setminus \{j\}$.  Thus, we have $x_{(i,j)}=1$ in \eqref{eqn: proj_LO} if and only if $i \,R\, j$ (and $x_{(i,j)}=0$ otherwise).  This establishes $\pi(\chi^{A(P)}) = v$.

Second, consider a vertex $\Phi$ of $F(\Dlo S)$. Then $\Phi$ is the characteristic vector of the set of arcs of an $\src$--$\snk$ path in $\Dlo S$. This path produces a profile that in turn represents a unique linear order $R$.  Letting $v$ denote the characteristic vector of $R$, we again have $\pi(\Phi) = v$.

The canonical description of the flow polytope $F(\Dlo S)$ has $n\,2^{n-1}$ inequalities.
\end{proof}

\begin{remark}
Recall that, for $n = 8$, the linear order polytope has at least $488{,}602{,}996$ facets. 
In contrast, the size of the canonical description of the corresponding flow polytope $F$ is only $1{,}024$.  
Even for $n=6$ or $n=7$, the minimal description of $P_{LO}^S$ has size larger than the canonical description of the extended formulation.
\end{remark}

\subsection{The case of the weak order polytope}\label{sub_weak_order}
For the weak orders of $S$, we define the network $\Dwo S = (N,A)$ by
\begin{eqnarray*}
N & := & \{X \st X \subseteq S\},\\ 
A & := & \{(X,Z)\in N \times N \st X \subset Z \}
\end{eqnarray*}
with the source and sink nodes $\src := \varnothing$ and $\snk := S$, respectively (i.e., the ``Boolean lattice'' $(\mathcal{P}(S),\subset)$ of $S$ with strict inclusion).
Clearly, we have $|N| = 2^n$. Because each arc $(X,Z)$ of $\Dwo S$ can be encoded as a word of size $n$ on the alphabet $\{a,b,c\}$ having at least one $b$ (the letters $a$, $b$ and $c$ correspond to alternatives of $X$, $Z \setminus X$ and $S \setminus Z$, respectively), we see that $|A| = 3^n - 2^n$. 

\begin{theorem}
\label{thm: proj_WO}
Let $\pi$ be the projection from $\mathbb{R}^{A}$ to $\mathbb{R}^{S \star S}$, mapping a point $\Phi$ from $\mathbb{R}^{A}$ to the point $x$ in $\mathbb{R}^{S \star S}$ given by Equation~\eqref{eqn: proj_LO} above. Then $\pi$ maps the flow polytope $F(\Dwo S)$ to the weak order polytope $\Pwo S$. This yields the extended formulation $F(\Dwo S)$ of $\Pwo S$, whose canonical description has size $3^n-2^n$.
\end{theorem}

\begin{proof}
The proof is parallel to that of Theorem~\ref{thm: proj_LO}, hence we provide only a sketch. 
Take any vertex $v$ of the weak order polytope $\Pwo S$, with $v$ encoding the weak order $R$.  Then $v$ is the image by $\pi$ of a well defined vertex $\chi^{A(P)}$ of the flow polytope $F(\Dwo S)$.  
To obtain $\chi^{A(P)}$, take $P$ to be the path in $(\mathcal{P}(S),\subset)$ resulting from the weak order $R$.  Here, for a given weak order $R$ and distinct alternatives $i$, $j$ in $X$, the special pair $(X,Z)$ in~\eqref{eqn: proj_LO} consists of the smallest set $Z$ in the profile of $R$ that contains $j$, while $X$ is just the set preceding $Z$ in the profile.
\end{proof}

\begin{remark}
Recall that, for $n = 5$, the weak order polytope $P_{WO}^S$ has $75{,}834$ facets \citep{Regenwetter_Stober_2012}.  In constrast, the size of the canonical description of the corresponding flow polytope $F(\Dwo S)$ is only $211$.
\end{remark}

\subsection{The case of the interval order polytope}

The network $\Dio S = (N,A)$ for the interval order polytope of $S$ is defined as follows:
\begin{eqnarray*}
N & := & \{(X,Y) \st Y \subseteq X \subseteq S\},\\[1mm] 
A & := & \left\{
((X,Y),(Z,T))\in N \times N 
\;\vrule height23pt width1pt depth14pt \; 
\begin{array}{l}
X \subseteq Z,\;
Y \subseteq T,\;\text{and}\\  
\begin{array}{rl}
\text{either~}&|Z| = |X| + 1,\; |T| = |Y|\\ 
\text{or~}&|Z| = |X|,\; |T| = |Y| + 1
\end{array}
\end{array}
\right\}.
\end{eqnarray*}
We choose the source and sink nodes to be $\src := (\varnothing,\varnothing)$ and $\snk := (S,S)$, respectively. 

\begin{lemma}\label{lem_Dio}
The numbers of nodes and arcs in the network $\Dio S$ are
\begin{equation}
|N| = 3^n \quad \textrm{and} \quad |A| = 2 \, n \, 3^{n-1},
\end{equation}
respectively.
\end{lemma}

\begin{proof}
Let $S = \{s_1,s_2,\ldots,s_n\}$. Consider a word $\omega = \omega_1\omega_2\cdots\omega_n$ of size $n$ on the alphabet $\{a,b,c\}$. To each such word $\omega$, we associate a pair of sets $(X,Y)$ such that $Y \subseteq X \subseteq S$ by
 letting $Y=\{s_i \in S \st \omega_i = a\}$ and $X=\{s_i \in S \st \omega_i \in \{a,b\}\}$.
We obtain a bijection from the set of words of size $n$ on the alphabet $\{a,b,c\}$ to the node set of the network $\Dio S$. Thus, we have $|N| = 3^n$. 

Now, consider a word $\omega = \omega_1\omega_2\cdots\omega_n$ of size $n$ on the alphabet $\{a,b,c$, $d,e\}$  
containing either exactly one $d$ and no $e$, or exactly one $e$ and no $d$.
We associate an arc $((X,Y),(Z,T))$ of the network $\Dio S$ to each such word, as follows. We let
\begin{eqnarray*}
X & := & \{s_i \in S \st \omega_i \in \{a,b,e\}\},\\
Y & := & \{s_i \in S \st \omega_i = a\},\\
Z & := & \{s_i \in S \st \omega_i \in \{a,b,d,e\}\},\\
T & := & \{s_i \in S \st \omega_i \in \{a,e\}\}.
\end{eqnarray*}
This gives a one-to-one correspondence between the words of size $n$ on the alphabet $\{a,b,c,d,e\}$ 
containing exactly one of the two letters $d$ and $e$,
and the arcs of $\Dio S$. We conclude that $|A| = 2 \, n\,3^{n-1}$. 
\end{proof}

The last result of this section is as follows.

\begin{theorem}
\label{thm: proj_IO}
Let $\pi$ be the projection from $\mathbb{R}^{A}$ to $\mathbb{R}^{S \star S}$, mapping a point $\Phi \in \mathbb{R}^{A}$ to the point $x \in \mathbb{R}^{S \star S}$ given,
 for all $(i,j) \in S \star S$,  by
\begin{equation}
\label{eqn: proj_IO}
x_{(i,j)} := \sum \{\Phi_a \st a = ((X,Y),(Z,T)) \in A,\ i \in Y,\ j \in Z \setminus X\}.
\end{equation}
Then $\pi$ maps the flow polytope of $\Dio S$ onto the interval order polytope of $S$.  Thus, the canonical description of $F(\Dio S)$ is an extended formulation of $\Pio S$ with description size $|A| = 2\,n\,3^{n-1}$.
\end{theorem}

\begin{proof}
The structure of the proof is similar to that of Theorem \ref{thm: proj_LO} (see however Remark~\ref{rmk: many-to-one} below). We show that the image of the vertex set of $F(\Dio S)$ under the projection $\pi$ is exactly the vertex set of $P_{IO}^S$. 

Consider a vertex $v = x^R$ of $P_{IO}^S$ and a vertex $\Phi = \chi^{A(P)}$ of $F(\Dio S)$ such that $\Phi$ encodes the profile of some ordinal representation of the interval order $R$, say by injective mappings $\ell$ and $h$ from $S$ to $\mathbb{R}$ with $\ell(S) \cap h(S) = \varnothing$. It suffices to prove that $\pi(\Phi) = v$ in this case (notice that each vertex of $F(\Dio S)$ comes from some profile of some interval order). 

Let $i$ and $j$ be two distinct alternatives of $S$. 
The right-hand side of \eqref{eqn: proj_LO} must give $Z=X+\{j\}$ and $Y=T$ (by the definition of the arcs in $A$).  In other words, given the profile, the arc $a = ((X,Y),(Z,T))$ is completely determined by $i$ and $j$.  Moreover, such an arc can exist in $A(P)$ only if we have $i \,R\, j$.  Hence, the right-hand side of \eqref{eqn: proj_LO} equals $1$ if and only if $i \,R\, j$ holds (and $0$ otherwise), thus $x_{(i,j)}=v_{(i,j)}$.
\end{proof}

\begin{remark}\label{rmk: many-to-one} 
In the proofs of Theorem~\ref{thm: proj_LO} and \ref{thm: proj_WO}, the projection of the flow polytope on the order polytope induces a bijective mapping between the vertex sets of the two polytopes.  This is not the case here: in general, several vertices of the flow polytope are mapped onto the same vertex of $P_{IO}^S$.  The reason is that different representations of the same interval order can lead to distinct profiles, and thus to different paths in the network.
\end{remark}

\begin{remark}
Recall that, for $n = 4$, the interval order polytope has $191$ facets. This is less than the size of the canonical description of the flow polytope $F(\Dio S)$ with $|S|=4$, which is $216$.   
For larger values of $n$, we do not know the exact number of facets of the interval order polytope.  However, a lower bound on this number follows from \citet[][Theorem~8 and Corollary~1]{Doignon_Rexhep_2016}: the interval order polytope $F(\Dio S)$ has at least as many facets as there are `PC-graphs' on $S$, where a \textsl{PC-graph} is a directed graph in which any node is the tail of at most one arc, and the head of at most one arc (in other words, a PC-graph is a node-disjoint union of paths and cycles plus maybe isolated nodes).  For small numbers of alternatives, we provide in Table~\ref{tab_io} both the resulting lower bound on the size of any linear description of the interval order polytope $P_{IO}^n$, and the exact size of the canonical description of the flow polytope $F(\Dio S)$.
The flow polytope definitely has a smaller description size for $n=6$ (the same could hold for $n=5$).  It is not difficult to see that the same assertion holds for $n\ge7$. 

\begin{table}
\begin{equation*}
\begin{array}{r@{\qquad}r@{\qquad}r}
n & \text{size of minimal } & \text{size of canonical }\\
  & \text{description of } P_{IO}^n & \text{description of } F(\Dio S)\\[2mm]\hline
  3 &          17 &    54 \rule{0pt}{3ex}\\
  4 &         191 &      216\\
  5 &     \ge 759 &  810\\
  6 & \ge 5{,}557 &  2{,}916\\
  7 &\ge 63{,}839 & 10{,}206
\end{array}
\end{equation*}
\caption{Comparison of description sizes in the case of interval orders.\label{tab_io}}
\end{table}
\end{remark}

\subsection{The case of the semiorder polytope}
\label{Semiorder polytopes}

The network $\Dso S = (N,A)$ we build for the semiorder polytope of $S$ is more structured than the one for the interval order polytope of $S$. The nodes of $\Dso S$ are of the form $(X,Y,L)$ where $X$ and $Y$ are subsets of $S$ such that $Y \subseteq X$, and $L$ is a linear order on $X\setminus Y$ 
(recall that when $X\setminus Y= \varnothing$ we have $L=\varnothing$).
Thus,
\begin{eqnarray*}
N &=& \{(X,Y,L) \st Y \subseteq X \subseteq S,\; L\textrm{ linear order on } X\setminus Y\}.
\end{eqnarray*}
We now indicate when the pair $((X,Y,L),(Z,T,M))$ of nodes from $\Dso S$ forms an arc.  
There are two cases that we illustrate in Figure~\ref{fig semiorder intervals} in terms of levels in representations: 
($\alpha$)~the threshold moves from $\thr=\ell(5)$ to $\thr = h(3)$; ($\beta$)~it moves from $\thr=h(3)$ to $\thr = h(4)$. 
Formally, the two cases are: 
\begin{enumerate}
\item[\qquad($\alpha$)]
for some alternative $i$ in $S \setminus X$:
\begin{eqnarray*}
Z &=& X \cup \{i\},\\ 
T &=& Y,\\
M &=& L + i,
\end{eqnarray*}
where $L + i$ means that we append $i$ at the end of the linear order $L$; 
\item[\qquad($\beta$)]
for the alternative $j$ in $X \setminus Y$ which is the first one in $L$:
\begin{eqnarray*}
Z &=& X,\\ 
T &=& Y \cup \{j\},\\
M &=& L - j,
\end{eqnarray*}
where $L - j$ means the linear order induced by $L$ on $Z \setminus T = (X\setminus Y) \setminus \{j\}$.
\end{enumerate}

The sink and source nodes are again $\src := (\varnothing,\varnothing)$ and $\snk := (S,S)$, respectively.

\begin{lemma}
The numbers of nodes and arcs in the network $\Dso S$ are
\begin{eqnarray} 
|N| &=& 
\sum_{t=0}^n\;n\,(n-1)\,\dots(n-t+1)\,2^{n-t},
\label{eq_|N|}
\\ 
|A| &=& 
\sum_{t=0}^n\; n\,(n-1)\,\dots(n-t+1)\,
2^{n-t-1}\,(n+t),
\label{eq_|A|}
\end{eqnarray}
respectively, and they satisfy
\begin{equation}
n! \le |N| \le e^2\,n!,
\qquad 
n! \le |A| \le e^2 \, n \, n!.
\end{equation}
\end{lemma}

\begin{proof}
To obtain a node $(X,Y,L)$ of $\Dso S$ with $t:=|X \setminus Y|$, we select first an ordered list of $t$ elements from $S$. Next we select some subset $Y$ among the $n-t$ remaining elements, and finally 
we set $X$ equal to $Y$ augmented by the listed elements.

To count the arcs, notice that the number of arcs leaving a node $(X,Y,L)$ equals $n-|Y|$.  Assume again $t=|X \setminus Y|$, and let $\ell:=|Y|$.  Then 
\begin{equation}
|A|=\sum_{t=0}^n\; \left(n\,(n-1)\,\dots(n-t+1)\,
\sum_{\ell=0}^{n-t} \binom{n-t}{\ell}\,(n-l) \right).
\end{equation}
Now
\begin{equation}
\sum_{\ell=0}^{n-t} \binom{n-t}{\ell}\,(n-l) = n\,2^{n-t}-(n-t)\,2^{n-t-1}
= 2^{n-t-1}\,(n+t).
\end{equation}
The first bound holds because the last term in Equation~\eqref{eq_|N|} is $n!$.  On the other hand, rewriting Equation~\eqref{eq_|N|} as
\begin{equation}
|N| \quad=\quad n! \; \sum_{t=0}^n \; \frac{2^{n-t}}{(n-t)!},
\end{equation}
we get $|N|\le n! \, e^2$. 
Because the number of arcs leaving any node is at most $n$, the upper bound on $|A|$ follows from the one on $|N|$.
\end{proof}

\begin{theorem}
\label{thm: proj_SO}
Let $\pi$ be the projection from $\mathbb{R}^{A}$ to $\mathbb{R}^{S \star S}$, mapping a point $\Phi \in \mathbb{R}^{A}$ to the point $x \in \mathbb{R}^{S \star S}$ given for all $i$, $j$ in $S$ by 
\begin{equation}
\label{eqn: proj_SO}
x_{(i,j)} := \sum \{\Phi_a \st a = ((X,Y,L),(Z,T,M)) \in A,\ i \in Y,\ j \in Z \setminus X\}.
\end{equation}
Then $\pi$ maps the flow polytope of $\Dso S$ to the semiorder polytope of $S$. Thus, the flow polytope
$F(\Dso S)$ is an extended formulation of the semiorder polytope $\Pso S$,
with size $|A| \le e^2 \, n \, n!$.
\end{theorem}

\begin{proof}
Once again, we show that $\pi$ maps the set of vertices of the flow polytope $F(\Dso S)$ onto the set of vertices of the order polytope, here the semiorder polytope $\Pso S$ of $S$.  Each $\src$--$\snk$ path in $\Dso S$ produces a representation of some interval order $R$, exactly as in the proof of Theorem~\ref{thm: proj_IO}.  Next, taking into account the ordering $L$ in any network node $(X,Y,L)$, we see that the representation thus obtained has no interval strictly including another one.  By Theorem~\ref{eqn: no embedded intervals}, $R$ is then a semiorder.  The rest of the proof is similar to the proof of Theorem~\ref{thm: proj_IO}.
\end{proof}

\begin{remark}
Recall that, for $n = 4$, the semiorder polytope has $563$ facets.  
The description size of the flow polytope $F(\Dso S)$ with $|S|=4$ is $520$, which is less than the number of facets of the semiorder polytope.
It is difficult to prove that this holds for all values of $n \ge 4$, because only a very limited number of facets of semiorder polytopes are known.
However we can prove that there exists a natural number $n_0$ such that, for all $n > n_0$, the description size of the flow polytope $F(\Dso S)$ will be smaller than the number of facets of the polytope $\Pso S$.
This is due to the fact that there are more facet defining inequalities for $\Pso n$ than PC-graphs on $n$ elements \cite[][Section~9]{Doignon_Rexhep_2016}, and more PC-graphs than \textsl{chain gangs}, the latter being the disjoint unions of paths plus maybe isolated nodes.  Indeed, the number of chain gangs on $n$ elements equals $n!\,\sum_{j=0}^{n-1}\, \binom{n-1}j\,/\, (j+1)!$ \cite[][Sequence~A000262]{Sloane_OEIS}, and so it becomes eventually larger than $e^2\,n\,n!$ for some $n=n_0$, and remains so for all $n$ larger than $n_0$\footnote{Computations suggest this happens with $n_0=21$.}.
Even if its description size happens to be larger for values of $n$ less than $n_0$, our extended formulation has always the  advantage that its canonical description is available in a very simple form.  
\end{remark}

\section{Interface between Flow Polytopes and Statistics}

In this section, we briefly consider how our network flow polytopes interface with contemporary statistical methods for evaluating  random preference models.  
Typical choice data used in evaluating random preference models are comprised of decision maker responses to a series of repeated paired presentations of choice alternatives \citep[see][]{Regenwetter_Dana_Stober_2011a}.  
We consider the case when pairs of options are offered and the decision maker must choose one (two-alternatives forced choice) or must either choose one or indicate indifference (ternary paired comparisons). 
Let $p_{(i,j)}$ and $C_{(i,j)}$ denote the probability and the number of times that a decision maker chooses alternative $j$ when offered  alternatives $i$ and $j$.   
Let $\boldsymbol{p} := \left(p_{(i,j)}\right)_{(i,j) \in S \star S}$ and let $\boldsymbol{C} := \left(C_{(i,j)}\right)_{(i,j) \in S \star S}$.    Let $L(\boldsymbol{p} | \boldsymbol{C})$ 
denote the likelihood of any given data $\boldsymbol{C}$ as a function of probabilities $\boldsymbol{p}$,  
\citep[for examples involving multinomial distributions see e.g.,][]{Davis_Stober_2009,Myung_Karabatsos_Iverson_2005, Regenwetter_Dana_Stober_2011a}.  
In general, in a random preference model, the maximum likelihood estimate does not have a closed form solution and therefore requires convex optimization.
The maximum likelihood estimate (MLE) for a given random preference model $P_{\mathcal M}^S$  that is one of $\Plo S$,  $\Pwo S$,  $\Pso S$, or  $\Pio S$, given choice data $\boldsymbol{C}$,  equals 
\begin{equation}
\hat{\boldsymbol{p}}_{\boldsymbol{\mathcal{M}}}^{\boldsymbol{C}} := \argmax_{ \boldsymbol{p} \in P_{\mathcal M}^S } L(\boldsymbol{p} | \boldsymbol{C}).
\end{equation}
In a network flow representation, the MLE for $P_{\mathcal M}^S$ can be calculated via the following optimization program:
\begin{equation}
\begin{array}{rrl}
\textrm{(MLE-CP)} &\min_{\boldsymbol{p}\in \mathbb{R}^{S \star S}, \Phi \in \mathbb{R}^{|A|}} & - \ln L(\boldsymbol{p} | \boldsymbol{C})\\
         &\textrm{s.t.} &\Phi \in F(D_{\mathcal M}^S)\\
        &              &\boldsymbol{p} = \pi(\Phi),
\end{array}
\end{equation}
where $D_{\mathcal M}^S$ is one of  $\Dlo S$,  $\Dwo S$,  $\Dso S$, or  $\Dio S$, accordingly. 
This is a convex optimization problem because we minimize the opposite of the log-likelihood function, assumed to be convex, over a convex feasible region.  Indeed, the two vectors of variables $\boldsymbol{p}$ and $\boldsymbol{\Phi}$ are only constrained by linear equations and inequalities, coming both from the canonical description \eqref{eqn: flow polytope} of the flow polytope $F(D_{\mathcal M}^S)$ and the expression that $\boldsymbol{p}$ is the projection of $\boldsymbol{\Phi}$ by $\pi$ (i.e. its image by an affine mapping).

The solution of this convex program yields the MLE as well as the maximized likelihood value. This convex program can be solved computationally by standard methods in convex optimization such as polynomial-time interior-point methods, see \cite{Nesterov_Nemirovskii_1994}. 
Alternatively, the vector of variables $\boldsymbol{p}$ can be eliminated from the formulation using the projection equalities, leading to
\begin{equation}
\begin{array}{rrl}
\textrm{(MLE-CP')} &\min_{\Phi \in \mathbb{R}^{|A|}} & - \ln L(\pi(\Phi) | \boldsymbol{C})\\
         &\textrm{s.t.} &\Phi \in F(D_{\mathcal M}^S).    
\end{array}
\end{equation}
This is now a convex nonlinear network optimization problem, see e.g.~\cite{Bertsekas_1998} 
(note that while the objective function remains convex, it is no longer separable). 

The maximum likelihood estimate $\hat{\boldsymbol{p}}_{\boldsymbol{\mathcal{M}}}^{\boldsymbol{C}}$ is obtained by projecting the optimum flow $\hat{\Phi}$ (which is not necessarily unique).

The Bayes factor is a standard method for assessing the relative empirical evidence for/against either of two competing models.
It is defined as the ratio of their respective marginal likelihoods \citep{kass1995bayes}.  Following previous approaches in testing random preference models \citep{davis2015individual}, we compare the model $P_{\mathcal M}^S$ to an ``encompassing''  model that places no restrictions on choice probabilities. Assuming suitably chosen prior distributions, 
the Bayes factor can be re-written as the ratio of two proportions: 
the proportion of the encompassing prior in agreement with $P_{\mathcal M}^S$ and the proportion of the encompassing posterior  in agreement with $P_{\mathcal M}^S$
\citep[see][]{klugkist2007bayes}.  We estimate these proportions by  repeatedly drawing from the prior and posterior distributions and counting the samples that satisfy $P_{\mathcal M}^S$.  
Checking whether or not a sampled value of $\boldsymbol{p}$  satisfies a given random preference model, using the extended formulation via a network flow, reduces to solving the following convex optimization program:

\begin{equation}
\begin{array}{rrl}
\textrm{(Bayes-CP)} &\min & \| \boldsymbol{p} - \boldsymbol{p}_{s}\|\\
         &\textrm{s.t.} &\Phi \in F(D_{\mathcal M}^S)\\
        &              &\boldsymbol{p} = \pi(\Phi),
\end{array}
\end{equation}
where $D_{\mathcal M}^S$ is one of  $\Dlo S$,  $\Dwo S$,  $\Dso S$, or  $\Dio S$, which can be solved in the same ways as (MLE-CP). If the optimal value of $\boldsymbol{p}$ equals $\boldsymbol{p}_{s}$ then the sampled point lies inside the 
model, otherwise it does not.  See \cite{davis2015individual} for additional details on using sampling methods to calculate Bayes factors under the weak order polytope.

\section{Optimality} \label{sec:optimality}

Here we discuss the asymptotic optimality of our extended formulations in terms of size. It will be convenient to use the notion of \textsl{extension complexity} of a polytope $P$, defined as the minimum size of an extended formulation of $P$. Notice that the extension complexity is affinely invariant, that is, two affinely equivalent polytopes $P$ and $Q$ have the same extension complexity, and that it is monotone in the sense that the extension complexity of a polytope $P$ is at least that of any of its faces $F$.

\cite{Maksimenko2017} recently proved that the \textsl{correlation polytope}
\begin{equation}
\Pcor{n} := \conv \{xx^\intercal \mid x \in \{0,1\}^n\}
\end{equation} 
is affinely equivalent to a face of the linear ordering polytope $\Plo{2n}$
(the linear ordering polytope for a set of $2n$ alternatives). This implies that the extension complexity of $\Plo{2n}$ is at least that of $\Pcor{n}$, which in turn was proved to be at least $(3/2)^n$, see \cite{FMPTW15jour} and \cite{KW14}. Therefore, the extension complexity of $\Plo{n}$ is at least $(3/2)^{\lfloor n/2 \rfloor}$. Moreover, since $\Pwo{n}$, $\Pio{n}$ and $\Pso{n}$ all have $\Plo{n}$ as a face, we see that their extension complexities are at least that of $\Plo{n}$.
Since the extended formulations constructed in this paper for $\Plo{n}$, $\Pwo{n}$ and $\Pio{n}$ have size $2^{\Theta(n)}$, this implies that the three corresponding extension complexities are also $2^{\Theta(n)}$, i.e.\ that these extended formulations are asymptotically optimal.

For the semiorder polytope, we only know that the extension complexity of $\Pso{n}$ is $2^{\Omega(n)}$ and at the same time $2^{O(n \log n)}$. We remark that it is possible that the semiorder polytope admits an extended formulation with size $2^{\Theta(n)}$ but no such extended formulation based on flows. Indeed, flow-based extended formulations are known to have strong limitations, see~\cite{Fiorini_Pashkovich2015}.

\section{Conclusions and Discussion}

In this paper, we leverage extended formulations and network flow polytopes to work with random preference models for linear orders, weak orders, semiorders, and interval orders. 
Our results break through previous barriers in the number of choice alternatives for which parsimonious linear descriptions are known.
One fundamental reason is that we provide a complete, linear description of the extended formulation (a polytope which projects on the polytope of choice probabilities permitted by the model).  A second reason is that in many cases the extended formulation description is of smaller size than that of the initial polytope.  However, in the case of semiorders, the extended formulation relies more heavily on the numerical representations of the relations 
and consequently entails an excessively large description.  
One interesting open question is to find a more parsimonious extended formulation in the case of semiorders, or prove that none exists.

A natural application of the extended formulations we have described is to expand empirical studies of  fundamental structures (e.g., weakly ordered preferences) for  larger numbers of choice alternatives than previously possible \citep{Regenwetter_Stober_2012}.  
We leave this for future work.

\vskip4mm

\textsc{Acknowledgments.}
This material is based upon work supported 
by the Air Force Office of Scientific Research under Award Nr.~FA9550-05-1-0356 (M.~Regenwetter, PI), 
by the  National Institutes of Health under Award K25AA024182 (C.~Davis-Stober, PI), 
by the National Institute of Mental Health under Award Nr.~PHS 2 T32 MH014257 (M.~Regenwetter, PI), 
by the National Science Foundation under awards SES-1062045, 
SES-1459699 (M.~Regenwetter, PI), and SES-1459866 (C.~Davis-Stober, PI). Jean-Paul Doignon and Samuel Fiorini were supported by \textsl{Action de Recherche Concert\'ee} grants of the \textsl{Communaut\'e fran\c{c}aise de Belgique (Belgium)}. Samuel Fiorini was also supported by ERC Consolidator Grant 615640-ForEFront. Any opinions, findings, and conclusions or recommendations expressed in this publication are those of the authors and do not necessarily reflect the views of their funding agencies or universities. 

\section{References}

\end{document}